\documentclass[11pt, leqno]{amsart}

\usepackage{mathrsfs}
\usepackage{graphicx}

\usepackage{amsfonts,delarray,amssymb,amsmath,amsthm,a4,a4wide}
\usepackage{latexsym}
\usepackage{epsfig}
\usepackage{color}

\newtheorem{thm}{Theorem}[section]
\newtheorem{cor}[thm]{Corollary}

\newtheorem{lem}[thm]{Lemma}
\newtheorem{prop}[thm]{Proposition}
\newtheorem{rem}[thm]{Remark}

\theoremstyle{definition}

\numberwithin{equation}{section}

\newcommand{\R}{\mathbb R}

\newcommand{\e}{\varepsilon}

\newcommand{\p}{\partial}
\newcommand{\dist}{\mbox{dist}\,}
\newcommand{\diam}{\mbox{diam}\,}

\newcommand{\comment}[1]{}
\def\h{\hspace*{.24in}}

\usepackage{esvect}

\makeatletter
\@namedef{subjclassname@2020}{%
  \textup{2020} Mathematics Subject Classification}
\makeatother
\begin{document} 

\title[Global Lipschitz estimates for the Monge--Amp\`ere eigenfunctions]{Global Lipschitz and Sobolev estimates for the Monge--Amp\`ere eigenfunctions of general bounded convex domains}
\author{Nam Q. Le}
\address{Department of Mathematics, Indiana University, 831 E 3rd St,
Bloomington, IN 47405, USA}
\email{nqle@iu.edu}
\thanks{The author was supported in part by the National Science Foundation under grants DMS-2054686 and DMS-2452320.}


\maketitle
\begin{abstract}

We show that the Monge--Amp\`ere eigenfunctions of general bounded convex domains are globally Lipschitz. The same result holds for 
 convex solutions to degenerate Monge--Amp\`ere equations of the form $\det D^2 u =M|u|^p$ with zero boundary condition on general bounded convex domains in $\R^n$ within the sharp threshold $p>n-2$. 
As a consequence, we obtain global $W^{2, 1}$ estimates for these solutions.

\medskip

\begin{center}
{\bf R\'esum\'e}
\end{center}

Nous montrons que les fonctions propres de l'op\'erateur de Monge--Amp\`ere des domaines convexes born\'es g\'en\'eraux sont globalement lipschitziennes. Le m\^eme r\'esultat est 
valable pour les solutions convexes des \'equations de Monge--Amp\`ere d\'eg\'en\'er\'ees de la forme $\det D^2 u =M|u|^p$ avec condition aux limites nulle sur les domaines convexes born\'es g\'en\'eraux de $\R^n$ dans le seuil pr\'ecis $p>n-2$.
En cons\'equence, nous obtenons des estimations globales dans $W^{2, 1}$ pour ces solutions.
 \end{abstract}
\section{Introduction and  statement of the main results}
\subsection{Motivations}
We are interested in obtaining global Lipschitz and Sobolev regularity 
for nonzero convex Aleksandrov solutions  $u\in C(\overline{\Omega})$ to the degenerate Monge--Amp\`ere equation
  \begin{equation}
  \label{up_eq}
   \det D^{2} u~=M |u|^p \h~\text{in} ~\Omega, \quad
u =0\h~\text{on}~\p \Omega,
\end{equation}
where $M>0$, $p> n-2$, and $\Omega\subset \R^n$ $(n\geq 2)$ is a general bounded convex domain that is not assumed to be smooth nor strictly convex.
The case $p=n$ corresponds to the Monge--Amp\`ere eigenvalue problem. When $p\neq n$, by multiplying $u$ with a suitable positive constant, we can always assume $M=1$. For $0<p\neq n$, the existence of a nonzero convex Aleksandrov solution to \eqref{up_eq} is guaranteed by \cite[Theorem 4.2]{LSNS} and it is unique when $p<n$. 

\medskip
Regarding interior regularity, solutions to \eqref{up_eq} satisfy $u\in C^{\infty}(\Omega)$; see, for example, \cite[Proposition 2.8]{LSNS} and \cite[Proposition 6.36]{Lbook}. On smooth and uniformly convex domains, the global smoothness of solutions to \eqref{up_eq} is by now well understood, thanks to \cite{HHW} in two dimensions and \cite{LS, S} in all dimensions; see also \cite{HL} for the global analyticity. 

\medskip
The eigenvalue problem for the Monge--Amp\`ere operator $\det D^2 u$ on uniformly convex domains $\Omega$ in $\R^n$ ($n\geq 2$) with  smooth boundary was 
first investigated by Lions \cite{Ln}. He showed that there exist a unique positive constant $\lambda=\lambda(\Omega)$ and a unique (up to positive multiplicative constants) nonzero 
convex function $u\in C^{1,1}(\overline{\Omega})\cap C^{\infty}(\Omega)$ solving the Monge--Amp\`ere eigenvalue problem: 
\begin{equation}\label{EPLi}
 \left\{
 \begin{alignedat}{2}
   \det D^2 u ~&=\lambda |u|^n \quad~&&\quad\text{in} ~\Omega, \\\
 u&= 0 \quad~&&\quad\text{on}~ \p \Omega.
 \end{alignedat}
 \right.
  \end{equation}
The constant $\lambda(\Omega)$ is called the Monge--Amp\`ere eigenvalue of $\Omega$. The nonzero convex solutions to \eqref{EPLi} are called the Monge--Amp\`ere eigenfunctions of $\Omega$.

\medskip
A variational characterization of  $\lambda(\Omega)$ was found by Tso \cite{Tso} who discovered that for uniformly convex domains $\Omega$ with sufficiently smooth boundaries, the following formula holds:
\begin{multline}
 \label{lamT}
 \lambda(\Omega)=\inf\Bigg\{ \frac{\int_{\Omega} |u|\det D^2 u\,dx}{\int_{\Omega}|u|^{n+1}\,dx}: u\in C^{0,1}(\overline{\Omega})\cap C^2(\Omega)\setminus\{0\},\\u~\text{is convex in } 
 \Omega,~u=0~\text{on}~\p\Omega\Bigg\}.
\end{multline}

\medskip
For general bounded convex domains $\Omega\subset\R^n$, the Monge--Amp\`ere eigenvalue problem \eqref{EPLi} was studied in \cite{LSNS} where 
 the existence, uniqueness and variational characterization of the Monge--Amp\`ere eigenvalue, and the uniqueness of the convex Monge--Amp\`ere eigenfunctions were obtained. In particular, there is a unique positive constant $\lambda=\lambda[\Omega]$ for problem \eqref{EPLi}  to have a nonzero convex solution $u\in C(\overline{\Omega})$, and the Monge--Amp\`ere eigenvalue $\lambda[\Omega]$ is characterized by 
 \begin{equation}
\label{lam_def}
\lambda[\Omega] =\inf\Bigg\{ \frac{\int_{\Omega} |u| \det D^2 u\,dx }{\int_{\Omega}|u|^{n+1}\, dx}: u\in C(\overline{\Omega})\setminus\{0\},
 \,u~\text{is convex in } \Omega,~ u=0~\text{on}~\p\Omega\Bigg\}.
 \end{equation}
 
 \medskip
Throughout, by an abuse of notation, we use $\det D^2 u~ dx$ to denote the Monge--Amp\`ere measure associated with a convex function $u$ on $\Omega$ (for which we refer to \cite[Chapter 3]{Lbook} for more details).
When $\Omega$ is uniformly convex with smooth boundary, by uniqueness, $\lambda(\Omega)$ defined by \eqref{lamT} and  $\lambda[\Omega]$ defined by \eqref{lam_def} must be equal, though the set of competitors for $\lambda(\Omega)$ is strictly contained in those for $\lambda[\Omega]$.  
We chose the bracket notation $\lambda[\Omega]$ in \eqref{lam_def} for the Monge--Amp\`ere eigenvalue of a general bounded convex domain $\Omega$ to emphasize that the boundary $\p\Omega$ might have flat parts or corners.

 \medskip
It was shown in \cite[Theorem 1.1]{LSNS} that the infimum in (\ref{lam_def}) is achieved by a 
 nonzero convex function $u\in C^{0,\beta}(\overline{\Omega})\cap C^{\infty}(\Omega)$ for all $\beta\in (0, 1)$ and the pair $(u,\lambda[\Omega])$ solves \eqref{EPLi}. This makes one wonder if one can take $\beta=1$ so as to extend Tso's characterization of the Monge--Amp\`ere eigenvalue from uniformly convex domains with smooth boundaries to general bounded convex domains. This boils down to the question of whether the Monge--Amp\`ere eigenfunctions of general bounded convex domains are globally Lipschitz. 
 We will answer positively this question in Theorem \ref{thm1}.
 Due to the degeneracy of \eqref{EPLi} or more general equation \eqref{up_eq} near the boundary, the global Lipschitz regularity is a very nontrivial issue, as reviewed below.

 \medskip

\begin{enumerate}
\item[$\bullet$]When $p<n-2$, \cite[Theorem 1.1]{LCPAA} shows that the unique nonzero convex solution $u$ to (\ref{up_eq}) has its gradient blowing up near any {\it flat part of the boundary}. 
However, for $0<p<n-2$, 
\cite[Proposition 1]{LDCDS} shows that for all $\beta \in (0, \frac{2}{n-p})$, one has
\begin{equation}
\label{ineqpn2}
|u(x)| \leq C(n, p,\beta, \Omega)[\dist(x,\p\Omega)]^{\beta} \quad\text{for all } x\in\Omega.
\end{equation}

\item[$\bullet$] When $n=2$ and $p=n-2=0$, the explicit solution to \eqref{up_eq}  on a planar triangle (see, for example, \cite[Section 2.2]{LCPAA} and \cite[Example 3.32]{Lbook}) shows that 
$u$ is only log-Lipschitz. The optimal boundary behavior for the nonzero convex solution to \eqref{up_eq} in the case $p=n-2>0$ has not been addressed in the literature.

\item[$\bullet$] When $p>n-2$, \cite[Corollary 1.5]{LCPAA} shows that convex solutions to \eqref{up_eq} are globally log-Lipschitz:
 \begin{equation}
 \label{logLip}
 |u(x)|\leq C(\Omega, n, p) \dist (x,\p\Omega) (1+ |\log \dist(x,\p\Omega)|) \|u\|_{L^{\infty}(\Omega)}\quad \text{for all }x\in\Omega.\end{equation} 
\end{enumerate}

 \subsection{Main results}
 Our first main theorem removes the logarithmic term in \eqref{logLip}, thereby establishing the global Lipschitz regularity for convex solutions to \eqref{up_eq} when $p>n-2$.
\begin{thm}[Global Lipschitz estimates for the Monge--Amp\`ere eigenfunctions] 
\label{thm1}
Let $\Omega\subset\R^n$ $(n\geq 2)$ be a bounded convex domain. Let $p>n-2$ and let $u\in C(\overline{\Omega})$ be a nonzero convex function. 
\begin{enumerate}
\item[(i)] Assume that $u$ solves the degenerate Monge--Amp\`ere equation
 \begin{equation*}
 \left\{
 \begin{alignedat}{2}
   \det D^{2} u~&=M |u|^{p} \h~&&\text{in} ~\Omega, \\\
u &=0\h~&&\text{on}~\p \Omega,
 \end{alignedat}
 \right.
\end{equation*}
where  $M>0$. Then $u$ is globally Lipschitz and we have the estimate
 \[|u(x)|\leq C(|\Omega|, \diam(\Omega), n, p, M) \dist (x,\p\Omega) \|u\|_{L^{\infty}(\Omega)}\quad \text{for all }x\in\Omega.\] 
 \item[(ii)] Assume that $u$ satisfies in the sense of Aleksandrov
 \begin{equation*}
   \det D^{2} u~\leq M [\dist(\cdot,\p\Omega)]^{p} \quad \text{in} ~\Omega, \quad
u =0\quad \text{on}~\p \Omega,
\end{equation*}
where  $M>0$. Then $u$ is globally Lipschitz and we have the estimate
 \[|u(x)|\leq C(\diam(\Omega), n, p, M) \dist (x,\p\Omega)\quad \text{for all }x\in\Omega.\] 
 \end{enumerate}
\end{thm}

\medskip
Note that the estimates in Theorem \ref{thm1} give the global Lipschitz property of $u$ due to its convexity. 
By \cite[Lemma 3.1 (iii)]{LSNS},  for $u$ satisfying Theorem \ref{thm1}(i), we have the estimates
\begin{equation}
\label{upinfty}
c(n, p)|\Omega|^{-2} \leq M\|u\|^{p-n}_{L^{\infty}(\Omega)} \leq C(n, p)|\Omega|^{-2}.\end{equation}
Thus, the appearance of $ \|u\|_{L^{\infty}(\Omega)}$ is only necessary in the case of the Monge--Amp\`ere eigenvalue problem.  

\medskip
We quickly mention some implications of Theorem \ref{thm1}.

\medskip

As discussed earlier, an immediate consequence of Theorem \ref{thm1} is the following characterization of the Monge--Amp\`ere eigenvalue $\lambda[\Omega]$ defined in \eqref{lam_def}.
\begin{cor} 
Let $\Omega\subset\R^n$ $(n\geq 2)$ be a bounded convex domain. Then the Monge--Amp\`ere eigenvalue $\lambda[\Omega]$ of $\Omega$ has also the following characterization:
\begin{equation*}\lambda[\Omega] 
 = \inf\Bigg\{ \frac{\int_{\Omega} |u|\det D^2 u\,dx}{\int_{\Omega}|u|^{n+1}\,dx}: u\in C^{0,1}(\overline{\Omega})\cap C^2(\Omega)\setminus\{0\}, u~\text{is convex in } 
 \Omega, u=0~\text{on}~\p\Omega\Bigg\}.
 \end{equation*}
\end{cor}

\medskip
The global Lipschitz estimates in Theorem \ref{thm1}(i) give global $W^{2, 1}$ estimates for the solutions. We use the Hilbert--Schmidt norm for matrices in this note.
\begin{cor}[Global $W^{2, 1}$ estimates for the Monge--Amp\`ere eigenfunctions]
\label{W21cor} Let $\Omega\subset\R^n$ $(n\geq 2)$ be a bounded convex domain. Let $p>n-2$ and let $u\in C(\overline{\Omega})\cap C^\infty(\Omega)$ be a nonzero convex function solving
 the degenerate Monge--Amp\`ere equation
 \begin{equation*}
 \left\{
 \begin{alignedat}{2}
   \det D^{2} u~&=M |u|^{p} \h~&&\text{in} ~\Omega, \\\
u &=0\h~&&\text{on}~\p \Omega,
 \end{alignedat}
 \right.
\end{equation*}
where  $M>0$. Then $D^2 u\in L^1(\Omega)$ with the estimate
\[\int_\Omega \|D^2 u\|\,dx\leq C(|\Omega|, \diam(\Omega), n, p, M) \|u\|_{L^{\infty}(\Omega)}.\] 
\end{cor}
The conclusion of Corollary \ref{W21cor} is false when $p<n-2$ and  $\Omega$ is a rectangular box in $\R^n$, as \cite[Proposition 3.1]{LPDEA} shows that in this situation, $D^2 u\not\in L^{\frac{n-p}{2(n-p)-2} +\e}(\Omega)$ for any $\e>0$, so $D^2 u\not\in L^1(\Omega)$. 

\medskip
Though simple, the proof of Corollary \ref{W21cor} uses interior $C^2$ regularity for solutions to our degenerate Monge--Amp\`ere equation. This interior second-order regularity or even interior $W^{2,1}$ regularity as in the nondegenerate case \cite{DPF, DPFS} is not known for the convex functions in Theorem \ref{thm1}(ii) so it is not clear if global $W^{2, 1}$ estimates are possible for them.

\medskip
Another consequence of Theorem \ref{thm1} is the eventual global Lipschitz regularity for an iterative scheme for the Monge--Amp\`ere eigenvalue problem with general initial data. 
\begin{cor}
 \label{IIScor}
Let $\Omega \subset \R^n$ ($n\geq 2$) be a bounded and convex domain. For a convex function $u$ on $\Omega$,  we define its Rayleigh quotient by
\begin{equation*}
R(u) = \frac{\int_{\Omega} |u|\det D^2 u~dx}{\int_{\Omega} |u|^{n+1}~dx}.
\end{equation*}
Let $u_0 \in C(\Omega)$ be a nonzero convex function on $\Omega$ with $0<R(u_0)<\infty$.
For $k \geq 0$, define the sequence $u_{k+1} \in C(\overline{\Omega})$ to be the convex Aleksandrov solutions of the Dirichlet problem 
\begin{equation}\label{IIS}
\begin{cases}
\det D^2u_{k+1} = R(u_k) |u_k|^n & \quad \text{in } \Omega, \\
u_{k+1} = 0 & \quad \text{on } \partial \Omega.
\end{cases}
\end{equation}
Then $u_{k}\in C^{2k-2}(\Omega)\cap C^{0, 1}(\overline{\Omega})$ for all $k\geq (n+4)/2$.
\end{cor}
The iterative scheme \eqref{IIS} was first introduced by Abedin and Kitagawa \cite{AK} to solve the Monge--Amp\`ere eigenvalue problem \eqref{EPLi}
for a large class of convex initial data $u_0$ satisfying $\det D^2 u_0\geq c_0>0$. 
The sequence $\{u_k\}_{k=1}^{\infty}$ is obtained by repeatedly inverting the Monge--Amp\`ere operator with Dirichlet boundary condition. 
The iterative scheme (\ref{IIS})  was showed in \cite[Theorem 1.4]{LArx} to converge to problem \eqref{EPLi} for all  convex initial data having finite and nonzero Rayleigh quotient. We refer to \cite{LLQ} for numerical analysis of  the scheme \eqref{IIS}.

\medskip
The threshold $p>n-2$ in Theorem \ref{thm1} is sharp in all dimensions as global Lipschitz estimates are not possible in general for \eqref{up_eq} when $p\leq n-2$. 
The case  $p<n-2$ was treated in \cite[Theorem 1.1]{LCPAA}.
Our next theorem shows that when $p=n-2$, the nonzero convex solution of \eqref{up_eq} can have log-Lipschitz type behavior near the flat part of the boundary.
\begin{thm} [Infinite boundary gradient for degenerate Monge--Amp\`ere equations]
\label{pn2}
Let $\Omega\subset\R^n$ $(n\geq 2)$ be a bounded convex domain.
Let $u\in C(\overline{\Omega})\cap C^{\infty}(\Omega)$ be the nonzero convex solution to
  \begin{equation*}
 \left\{
 \begin{alignedat}{2}
   \det D^{2} u~&= |u|^{n-2} \h~&&\text{in} ~\Omega, \\\
u &=0\h~&&\text{on}~\p \Omega.
 \end{alignedat}
 \right.
\end{equation*}
\begin{enumerate}
\item[(i)] We have the estimate
 \[|u(z)|\leq C(\diam(\Omega), n) \dist (z,\p\Omega) \Big(1+ |\log \dist(z,\p\Omega)|^{\frac{n}{2}}\Big) \quad \text{for all }z\in\Omega.\] 
\item[(ii)] Assume that there is a nonempty closed subset $\Gamma$ of the boundary $\p\Omega$ that lies on a hyperplane. Then, for $x\in\Omega$ sufficiently close to the interior of $\Gamma$, we have 
 \[|u(x)| \geq c(n, \Omega,\Gamma)\dist (x,\p\Omega)\big|\log \dist (x,\p\Omega)\big|^{\frac{1}{n}}. \]
 \end{enumerate}
\end{thm}
It would be interesting to bridge the gap between the exponent $n/2$ and $1/n$ in the upper bound and lower bound in Theorem \ref{pn2}.

\subsection{On the proofs}
We say a few words about the proof of Theorem \ref{thm1}  in Section \ref{Lip_sec}. A key ingredient is the construction in Lemma \ref{Lipsub} of Lipschitz convex subsolutions $v$ to degenerate equations of the form $\det D^2 v \geq |v|^p$ where $p>n-2$ and $v$ is nonpositive on the boundary. 
If $\det D^2 u=M|u|^p$ and $p\geq n$, then $\det D^2 u\leq M\|u\|_{L^{\infty}(\Omega)}^{p-n+1}|u|^{n-1}$, so we can reduce the exponent $p$ to be below $n$, the critical exponent. 
Our interest here is the range
$p\in (n-2, n)$.  From Lemma \ref{Lipsub}, we obtain a direct proof of Theorem \ref{thm1} based on
a comparison principle for degenerate subcritical Monge--Amp\`ere equations $\det D^2 w=|w|^q$ ($0\leq q<n$) in Lemma \ref{comp_lem}.  This lemma allows us to improve the exponent $\beta$ in \eqref{ineqpn2} to $2/(n-p)$ when $p\in (0, n-2)$,  avoiding the iteration argument in  \cite{LDCDS} using $C^{\alpha}$ ($0<\alpha<1$) convex subsolutions of the form 
\[w_{\alpha}(x)=x_n^\alpha (|x'|^2-C_\alpha),\quad C_\alpha>0\] 
on the upper half space $\{x_n>0\}$. These functions were motivated by \cite[Lemma 1]{C} where Caffarelli considered the case $\alpha=2/n$ when $n\geq 3$ and $\alpha\in (0, 1)$ when $n=2$.

\medskip
One cannot take $\alpha\geq 1$ in $w_\alpha$ due to the convexity requirement; see \eqref{walsub}. Our crucial observation from \eqref{walsub} is that we can actually take $\alpha>1$ as long as we change the sign in front of $C_\alpha$, that is, we can look for convex subsolutions of the form $x_n^a (|x'|^2 + A)$ where $a>1$ and $A>0$.  The only issue with this ansatz is that it is always positive. However, this can be handled by subtracting from it a large multiple of $x_n$ to obtain desired globally Lipschitz convex subsolutions \[v(x) =x_n^a (|x'|^2 + A)-Bx_n\] to degenerate equations of the form $\det D^2 v \geq |v|^p$ where $p>n-2$ and $v$ is nonpositive on the boundary. Though simple, this final ansatz has escaped our attention up to now.
Finally,  the Lipschitz subsolutions allow us to obtain optimal boundary estimates for 
the Abreu's equation \cite{Ab} with degenerate boundary data.

\medskip
Throughout, we denote $x\in \R^n$ by $x=(x', x_n)$ and $\R^n_+:=\{x=(x', x_n)\in\R^n: x_n>0\}$. 
We use $\dist(\cdot, E)$ to denote the distance function to a closed set $E\subset\R^n$, $\diam(\Omega)$ for the diameter of a set $\Omega\subset\R^n$, and $|\Omega|$ for the $n$-dimensional Lebesgue measure of an open set $\Omega\subset\R^n$.
We use $C= C(\ast, \cdots, \star)$ to denote a positive
constant $C$ depending on the quantities $\ast, \cdots, \star$
in the parentheses.  In general, $C$ can be computed explicitly, and its value
may change from line to line in a given context.

\medskip
The rest of this note is organized as follows. We prove a comparison principle for degenerate subcritical Monge--Amp\`ere equations in Section \ref{comp_sec}.
The construction of globally Lipschitz subsolutions and applications including the proofs of Theorem \ref{thm1} and Corollaries \ref{W21cor} and \ref{IIScor} will be given in Section \ref{Lip_sec}. 
We prove Theorem \ref{pn2} in Section \ref{pn2_sec}.

\section{A comparison principle for degenerate subcritical Monge--Amp\`ere equations}
\label{comp_sec}
The following comparison principle for degenerate subcritical Monge--Amp\`ere equations is motivated from a uniqueness result in Tso \cite[Proposition 4.1]{Tso}. The supersolution here is not required to be convex. 
\begin{lem}[Comparison principle for degenerate subcritical Monge--Amp\`ere equations]
 \label{comp_lem} Let $p\in [0, n)$, $\delta\in [0,\infty)$, and $\Omega$ be a bounded convex domain in $\R^n$. Let $u, v\in C(\overline{\Omega})\cap C^2(\Omega)$ where 
 \begin{enumerate}
 \item[$\bullet$] $v$ is convex in $\Omega$, $v<0$ in $\Omega$, $v\leq 0$ on $\p\Omega$,  and $v$ is a subsolution in the sense that
 \begin{equation}
 \label{1vsub}
  \det D^{2} v\geq (|v|+\delta)^{p} \h~\text{in} ~\Omega,\end{equation}
 \item[$\bullet$] $u\geq 0$ on $\p\Omega$, and $u$ is a supersolution in the sense that
 \begin{equation}\label{2usup} \det D^{2} u\leq (|u|+\delta)^{p} \h~\text{in} ~\Omega.\end{equation}
 \end{enumerate}
 Then 
 \[u\geq v\quad\text{in }\Omega.\]
 Consequently, if instead of \eqref{1vsub} and \eqref{2usup}, we have
 \begin{equation}
 \label{3uv}
 \det D^2 v\geq K |v|^p \quad\text{and}\quad  \det D^2 u\leq L |u|^p\quad \text{in }\Omega,\end{equation} where $K, L>0$, then 
 \[-u \leq (K/L)^{\frac{1}{p-n}}|v| \quad\text{in }\Omega.\]
  \end{lem}
\begin{proof} We first prove $u\geq v$ in $\Omega$ when  \eqref{1vsub} and \eqref{2usup} hold.
  Suppose that  $v-u$ is positive somewhere in $\Omega$.  Note that $v-u\leq 0$ on $\p\Omega$.
By translation of coordinates, we can assume that $0\in\Omega$ satisfies
$v(0)- u(0)=\max_{\overline{\Omega}} (v-u) =\tau>0.$
For $1<\gamma\leq 2$, consider for $x\in\Omega$, \[v_{\gamma}(x)= v(x/\gamma),\quad
\text{and }\eta_{\gamma}(x) = u(x)/v_{\gamma}(x).\] If $\dist(x,\p\Omega)\rightarrow 0$, then $\limsup u(x)\geq 0$, so $\liminf \eta_{\gamma}(x)\leq 0.$
Since $v<0$ in $\Omega$, we have  $$\eta_\gamma(0) = u(0)/v(0) =[v(0)-\tau]/v(0) \geq 1+\e\quad \text{for some }\e>0.$$
Therefore, the function $\eta_{\gamma}$ attains its maximum value in $\overline{\Omega}$ at $x_{\gamma}\in\Omega$ with $\eta_{\gamma}(x_{\gamma})\geq 1+\e.$
At $x=x_{\gamma}$, we have $D\eta_{\gamma} (x_\gamma)=0$, $D^2 \eta_{\gamma }(x_{\gamma})\leq 0$, and $v_{\gamma}(x_{\gamma})< 0$, so we can compute
\[D^2 u(x_\gamma)= \eta_{\gamma} (x_\gamma) D^2 v_{\gamma} (x_\gamma) + D^2 \eta_{\gamma }(x_\gamma) v_{\gamma}(x_\gamma)\geq \eta_{\gamma}(x_\gamma) D^2 v_{\gamma}(x_\gamma),\] in the sense of positive definite matrices. 
Hence, the convexity of $v_\lambda$ and \eqref{1vsub} and \eqref{2usup} give
\begin{equation}
\label{etaD2}
\begin{split}
(|u(x_{\gamma})|+\delta)^p \geq \det D^2 u(x_{\gamma})& \geq [\eta_{\gamma}(x_{\gamma})]^n\det D^2 v_{\gamma}(x_{\gamma}) \\&\geq  [\eta_{\gamma}(x_{\gamma})]^n \gamma^{-2n} (|v(x_{\gamma}/\gamma)|+\delta)^p.
\end{split}
\end{equation}
Using $\eta_{\gamma}(x_{\gamma})\geq 1+\e>1$ and $\delta\geq 0$, we find 
\begin{equation}
\label{etaest}
|u(x_{\gamma})|+\delta = \eta_\gamma(x_\gamma)|v(x_{\gamma}/\gamma)|+\delta\leq \eta_\gamma (x_\gamma) (|v(x_{\gamma}/\gamma)|+\delta).\end{equation}
We deduce from the estimates \eqref{etaD2} and \eqref{etaest} that 
\[1\geq  [\eta_{\gamma}(x_{\gamma})]^{n-p} \gamma^{-2n} \geq (1+\e)^{n-p}\gamma^{-2n}.\]
Letting $\gamma\searrow 1$, using $p<n$ and $\e>0$, we obtain a contradiction. Therefore, $u\geq v$ in $\Omega$.

For the consequence, we rescale and take advantage of the subcriticality. Assume now that instead of \eqref{1vsub} and \eqref{2usup}, we have \eqref{3uv}.
Let $s_K= K^{\frac{1}{p-n}}$ and $s_L= L^{\frac{1}{p-n}}$. Then
\[\det D^2 (s_K v) \geq |s_K v|^p\quad\text{and } \det D^2 (s_L u) \leq |s_L u|^p\quad \text{in }\Omega.\]
Thus, as above for the case $\delta=0$, we have $s_Lu\geq s_Kv$ in $\Omega$. Hence, $-u \leq (K/L)^{\frac{1}{p-n}}|v|$ in $\Omega$, completing the proof of the lemma.
\end{proof}
As an application of Lemma \ref{comp_lem}, we improve upon \eqref{ineqpn2} where $\beta$ is allowed to be $2/(n-p)$.
\begin{prop}
Let $\Omega\subset\R^n$ $(n\geq 3)$ be a bounded convex domain. Let $p\in (0, n-2)$ and $M>0$. Let $u\in C(\overline{\Omega})$ be the nonzero convex solution to 
\[\det D^2 u =M |u|^p\quad\text{in }\Omega, \quad u=0\quad\text{on }\p\Omega. \]
Then
\[ |u(z)|\leq    C(n, p, M, \diam(\Omega))[\dist(z, \p\Omega)]^{\frac{2}{n-p}} \quad\text{for all } z\in\Omega.\]
  \end{prop}
  \begin{proof} Note that $u<0$ in $\Omega$. By \cite[Proposition 2.8]{LSNS}, $u\in C^{\infty}(\Omega)$. 
  Let $z=(z', z_n)$ be an arbitrary point in $\Omega$. By translation and rotation of coordinates, we can assume that:  $0\in \p\Omega$,  $\Omega\subset \R^n_+$, 
 the $x_n$-axis points inward $\Omega$, $z$ lies on the $x_n$-axis so $z=(0, z_n)$, and $z_n=\dist (z,\p\Omega)$.  
 Let $D:=\diam(\Omega)$.  Consider for $\alpha =\frac{2}{n-p} \in [\frac{2}{n}, 1)$
\[w_{\alpha}(x)= 
x_n^{\alpha} (|x'|^2 -C_\alpha)\quad\text{where } C_{\alpha}= (1+ 2D^2)/[\alpha(1-\alpha)].
\]
Then,  we find from \cite[Lemma 2.2]{LDCDS} that $w_{\alpha}$ is convex in $\Omega$, $w_{\alpha}< 0$ in $\Omega$, $w_{\alpha}\leq 0$ on $\p\Omega$, and   
\begin{equation}
\label{walsub}
\det D^2w_{\alpha}= 2^{n-1}x_n^{n\alpha-2} [\alpha(1-\alpha) C_\alpha- (\alpha^2+ \alpha) |x'|^2] \geq x_n^{n\alpha-2} \geq C_\alpha^{-\frac{n\alpha-2}{\alpha}} |w_\alpha|^{\frac{n\alpha-2}{\alpha}}\quad \text{in}~\Omega.
\end{equation}
The choice of $\alpha$ gives
\[\det D^2w_{\alpha} \geq C_{\alpha}^{-p}|w_\alpha|^p \quad \text{in}~\Omega.\]
As a consequence of the comparison principle in Lemma \ref{comp_lem}, we have 
\[|u(z)| \leq C(M, \alpha, n, p, D)|w_\alpha(z)| \leq  C(M, n, p, D) z_n^\alpha = C(n, p, M, D)[\dist(z, \p\Omega)]^{\frac{2}{n-p}}.\]
This proves the proposition. 
\end{proof}

\section{Globally Lipschitz subsolutions and applications}
\label{Lip_sec}
In this section, we prove Theorem \ref{thm1} and Corollaries \ref{W21cor} and \ref{IIScor}.

 We first construct globally Lipschitz convex subsolutions to the Monge--Amp\`ere equations $\det D^2 v= C|v|^{p}$ and $\det D^2 v= C[\dist(\cdot, \p\Omega)]^{p}$ where $p>n-2$ with nonpositive boundary values. 
\begin{lem}[Globally Lipschitz convex subsolutions]
\label{Lipsub}
Let $\Omega\subset\R^n_+$ $(n\geq 2)$ be a bounded convex domain with $0\in\p\Omega$. Let $D=\diam(\Omega)$.
For $a>1$, let
\[v(x)=v_{a, D}(x):= x_n^a (|x'|^2 + A) -B x_n,\quad\text{for } x=(x', x_n)\in\overline{\Omega},\]
where
\[A=A(a, D):= \frac{1+ (a+1)D^2}{ a-1},\quad B=B(a, D): = D^{a-1} (A+ D^2).\]
Then $v\in C^{0, 1}(\overline{\Omega})\cap C^{\infty}(\Omega)$, $v$ is convex in $\Omega$, $v\leq 0$ in $\overline{\Omega}$, and
\[\det D^2 v(x)\geq  2^{n-1}a x_n^{na-2} \geq \frac{2^{n-1}a}{B^{na-2}} |v|^{na-2} \quad\text{in }\Omega.\]
\end{lem}
\begin{proof} It is clear that $v\in C^{0, 1}(\overline{\Omega})\cap C^{\infty}(\Omega)$.
We have
\begin{equation*}
 D^2 v (x)=
  \begin{pmatrix}
  2x^a_n & 0 & \cdots & 0 & 2ax_1x_n^{a-1} \\
  0 & 2x^a_n & \cdots & 0 & 2ax_2x_n^{a-1} \\
  \vdots  & \vdots  & \ddots & \vdots & \vdots  \\
  0 & 0 & \cdots & 2x^a_n &  2ax_{n-1}x_n^{a-1}\\
  2ax_1 x_n^{a-1} & 2ax_2 x_n^{a-1} & \cdots & 2ax_{n-1}x_n^{a-1} & a(a-1) x_n^{a-2} (|x'|^2+ A) 
 \end{pmatrix}.
\end{equation*}
Note that the first $(n-1)$-leading principal minors of $D^2 v$ are all positive in $\Omega$.
By induction, we find that 
\begin{equation}
\label{vsub}
\begin{split}\det D^2 v(x)&= 2^{n-1} a(a-1) x_n^{na-2} (|x'|^2 + A)- 2^n a^2 |x'|^2 x_n^{na-2}\\ &=  2^{n-1}a x_n^{na-2}\big[ A(a-1) - (a+ 1)|x'|^2\big].\end{split}
\end{equation}
With the choices of $A$ and $B$,  we have $v$ is convex in $\Omega$, and $v\leq 0$ in $\overline{\Omega}$. 
Clearly, \[A(a-1) - (a+ 1)|x'|^2 \geq A(a-1)-(a+ 1)D^2\geq 1  \quad\text{in }\Omega,\] so \eqref{vsub} gives
\[\det D^2 v(x)\geq  2^{n-1}a x_n^{na-2}\geq \frac{2^{n-1}a}{B^{na-2}} |v|^{na-2}  \quad\text{in }\Omega.\]
The lemma is proved.
\end{proof}

\begin{proof}[Proof of Theorem \ref{thm1}]

Let $z=(z', z_n)$ be an arbitrary point in $\Omega$. By translation and rotation of coordinates, we can assume that:  $0\in \p\Omega$,  $\Omega\subset \R^n_+$, 
 the $x_n$-axis points inward $\Omega$, $z$ lies on the $x_n$-axis so $z=(0, z_n)$, and $z_n=\dist (z,\p\Omega)$.  
 
\medskip
We prove part (i). Note that the function $u$ in this part satisfies $u\in C^{\infty}(\Omega)$ by \cite[Proposition 2.8]{LSNS}. We  need to show that \begin{equation}
\label{uLip}
|u(z)| \leq C(p, n, M, |\Omega|,\diam(\Omega)) \dist (z,\p\Omega)= C(p, n, M, |\Omega|,\diam(\Omega)) z_n. \end{equation}

  First, we consider the case $n-2<p<n$. Let 
 \begin{equation}
 \label{aDv}
 a= \frac{2+ p}{n}>1, \quad D=\diam (\Omega), \quad v=v_{a, D},\end{equation}
 where $v_{a, D}(x)= x_n^a (|x'|^2 + A) -B x_n$ is the convex function in Lemma \ref{Lipsub}. Recall that \[\det D^2 v\geq c(a, n, D)|v|^{na-2}= c(a, n, D)|v|^p,\,\,  \det D^2 u\leq M|u|^p\quad \text{in }\Omega;\,\, v\leq u=0 \quad \text{on }\p\Omega.\]
 From the consequence of the comparison principle in Lemma \ref{comp_lem}, we have 
\[|u|\leq C(a, D, n, p, M)|v_{a, D}| \quad\text{in }\Omega.\] 
 Since $|v(z)|= |v_{a, D}(z)|\leq C(a, D)z_n$,
 we obtain \eqref{uLip}.
 
 \medskip
   Now, we consider the case $p\geq n$. If $p=n$, then by homogeneity, we can assume that $\|u\|_{L^{\infty}(\Omega)}=1$.  If $p>n$, we use \eqref{upinfty} to estimate $\|u\|_{L^{\infty}(\Omega)}$. In all cases, we have
  \[\det D^2 u =M|u|^p \leq M \|u\|^{p-n+1}_{L^{\infty}(\Omega)} |u|^{n-1}\leq C(M, n, p, |\Omega|) |u|^{n-1}\quad\text{in }\Omega,\]
  and \eqref{uLip} follows from the case $p=n-1$.  
  Part (i) is proved.
  
  \medskip
  Now, we prove part (ii). We  need to show that \begin{equation}
\label{uLip2}
|u(z)| \leq C(p, n, M, \diam(\Omega)) \dist (z,\p\Omega)= C(p, n, M, \diam(\Omega)) z_n. \end{equation}
Let \[a= (2+ p)/n>1, D=\diam (\Omega), \quad \text{and }v(x)=v_{a, D}(x)= x_n^a (|x'|^2 + A) -B x_n\] be the convex function in Lemma \ref{Lipsub}. Let $s= (M2^{1-n}/a)^{1/n}$. From Lemma \ref{Lipsub}, we have
\[\det D^2 (sv)= s^n \det D^2 v \geq s^n  2^{n-1}a x_n^{na-2} =Mx_n^p. \]
Thus
\[\det D^2 (sv) \geq M[ \dist (\cdot,\p\Omega)]^p \geq  \det D^2 u \quad \text{in }\Omega. \]
Note that $sv\leq u=0$ on $\p\Omega$.
By the comparison principle for the Monge--Amp\`ere equation (see \cite[Theorem 3.21]{Lbook}), we have 
\[sv\leq u \quad \text{in }\Omega.\]
In particular, at $z=(0, z_n)$, we have
\[|u(z)\leq |s v(z)| \leq sB(a, D) z_n.\]
This implies \eqref{uLip2}, completing the proof of the theorem.
  \end{proof}
  
  \begin{proof} [Proof of Corollary \ref{W21cor}]
By the global Lipschitz estimates in Theorem \ref{thm1}(i) and the convexity of $u$, we have
\begin{equation}
\label{Dubd}
\|Du\|_{L^{\infty}(\Omega)} \leq C(|\Omega|, \diam(\Omega), n, p, M) \|u\|_{L^{\infty}(\Omega)}.\end{equation}
The convexity of $u$ also implies that
\[\|D^2 u\| \leq \Delta u\quad\text{in }\Omega.  \]
 Let $\{\Omega_m\}_{m=1}^{\infty}\subset \Omega$ be a sequence of smooth convex domains that converges to $\Omega$ in the Hausdorff distance.  
Let $\mathcal{H}^{n-1}$ denotes the $(n-1)$-dimensional Hausdorff measure. 

For each $m$, integrating by parts gives
\[\int_{\Omega_m} \|D^2 u\|\, dx \leq \int_{\Omega_m} \Delta u\, dx =\int_{\p\Omega_m} \frac{\p u}{\p\nu}\, d \mathcal{H}^{n-1} \leq \mathcal{H}^{n-1} (\p\Omega_m) \|Du\|_{L^{\infty}(\Omega)},\]
where $\nu$ is the outer normal unit vector field on $\p\Omega_m$.

From \eqref{Dubd}, we easily obtain for all $m$
\[\int_{\Omega_m} \|D^2 u\|\, dx \leq  C(|\Omega|, \diam(\Omega), n, p, M) \|u\|_{L^{\infty}(\Omega)}.\]
Now, we let $m\to\infty$. The monotone convergence theorem then gives $D^2 u\in L^1(\Omega)$ with estimate for $\|D^2 u\|_{L^1(\Omega)}$ stated in the corollary. 
\end{proof}
  
  \begin{proof}[Proof of Corollary \ref{IIScor}] The interior regularity of the scheme $u_k\in C^{2k-2}(\Omega)$ for $k\geq 2$ follows from \cite[Proposition 3.1]{LArx}. In view of Step 1 in the proof of \cite[Theorem 1.4]{LArx}, the sequence $\{R(u_k)\}_{k=0}^{\infty}$ is bounded: 
  \begin{equation}\label{Rbd}R(u_k) \leq C(u_0, n,\Omega)\quad\text{for all } k\geq 0.\end{equation}
  The finiteness of $R(u_0)$ implies that of $\|u_0\|_{L^{n+1}(\Omega)}$. Hence
  \[\int_\Omega \det D^2 u_{1}\,dx =R(u_0)\int_\Omega |u_0|^n\, dx \leq C(u_0, n,\Omega). \] 
  As a consequence of the Aleksandrov estimate (see \cite[Theorem 3.12]{Lbook}), we have 
  \begin{equation}\label{1nH}|u_1(x)| \leq C(u_0, n,\Omega) [\dist(x,\p\Omega)]^{\frac{1}{n}}\quad\text{for all } x\in\Omega.\end{equation}
 \medskip
 {\bf Step 1.} We will show by induction that if $1\leq k\leq n/2$, then 
    \begin{equation}
    \label{2k1nH}
    |u_k(x)| \leq C(u_0, n,\Omega) [\dist(x,\p\Omega)]^{\frac{2k-1}{n}}\quad\text{for all } x\in\Omega.\end{equation}
    
    When $k=1$,  \eqref{2k1nH} is exactly \eqref{1nH}. Assume \eqref{2k1nH} holds for $1\leq k\leq (n-2)/2$ (if $n\leq 3$, then we are done). We prove it for $k+1$.  Indeed, recalling \eqref{Rbd} and \eqref{2k1nH}, we have
    \begin{equation}
    \label{uk1bd}
    \det D^2 u_{k+1} = R(u_k) |u_k|^n \leq C_k(u_0, n,\Omega ) [\dist(\cdot,\p\Omega)]^{2k-1}.\end{equation}
    Let $z=(z', z_n)$ be an arbitrary point in $\Omega$. By translation and rotation of coordinates, we can assume that:  $0\in \p\Omega$,  $\Omega\subset \R^n_+$, 
 the $x_n$-axis points inward $\Omega$, $z$ lies on the $x_n$-axis so $z=(0, z_n)$, and $z_n=\dist (z,\p\Omega)$.  
 Let $D:=\diam(\Omega)$.  Consider for $\alpha =\frac{2k+1}{n} \in [\frac{3}{n}, 1)$
\[w_{\alpha}(x)= 
x_n^{\alpha} (|x'|^2 -C_\alpha)\quad\text{where } C_{\alpha}= (1+ 2D^2)/[\alpha(1-\alpha)].
\]
Then,   $w_{\alpha}$ is convex in $\Omega$, $w_{\alpha}\leq 0$ on $\p\Omega$, and recalling \eqref{walsub}, we have 
\[
\det D^2w_{\alpha} \geq x_n^{n\alpha-2} \geq [\dist(\cdot,\p\Omega)]^{2k-1} \quad \text{in}~\Omega.
\]
It follows from \eqref{uk1bd} that 
\[\det D^2(C_k^{1/n}w_{\alpha}) \geq C_k[\dist(\cdot,\p\Omega)]^{2k-1}\geq \det D^2 u_{k+1} \quad \text{in}~\Omega.\]
The comparison principle for the Monge--Amp\`ere equation (see \cite[Theorem 3.21]{Lbook}) gives
\[C_{k}^{1/n} w_\alpha \leq u_{k+1} \quad \text{in }\Omega.\]
In particular, at $z=(0, z_n)$, we have
\[|u_{k+1}(z)|\leq C_{k}^{1/n} |w_\alpha(z)| \leq C(u_0, n, \Omega)z^{\frac{2k+1}{n}}_n= C(u_0, n,\Omega)  [\dist(z,\p\Omega)]^{\frac{2k+1}{n}}.\]
The arbitrariness of $z$ proves \eqref{2k1nH} for $k+1$. Therefore,   \eqref{2k1nH} holds for all $1\leq k\leq n/2$.

\medskip
\noindent
{\bf Step 2.} We now deduce the corollary from \eqref{2k1nH}. Indeed, for $m:= \lfloor n/2 \rfloor$, using \eqref{Rbd} and \eqref{2k1nH}, we find
\[\det D^2 u_{m+1} = R(u_{m})|u_{m}|^n \leq C [\dist(\cdot,\p\Omega)]^{2m-1} \leq  C(u_0, n,\Omega) [\dist(\cdot,\p\Omega)]^{n-5/2}.\]
The proof of \eqref{2k1nH} in fact shows that  if $\det D^2 u_k \leq C[\dist(\cdot,\p\Omega)]^{\alpha}$ where $\alpha < n-2$,  then the modulus of $u_{k}$ grows less than  $C[\dist(\cdot,\p\Omega)]^{\frac{2+\alpha}{n}}$.  Thus, we have
\[|u_{m+1}| \leq C(u_0, n,\Omega) [\dist(\cdot,\p\Omega)]^{\frac{n-1/2}{n}} \leq   C(u_0, n,\Omega) [\dist(\cdot,\p\Omega)]^{\frac{n-1}{n}} \quad\text{in }\Omega.\]
Therefore
\[\det D^2 u_{m+2} = R(u_{m+1})|u_{m+1}|^n \leq C [\dist(\cdot,\p\Omega)]^{n-1}. \]
By Theorem \ref{thm1}(ii) for $p=n-1$, we have $u_{m+2}\in C^{0, 1}(\overline{\Omega})$ with the estimate
\[|u_{m+2}| \leq C(u_0, n,\Omega) \dist(\cdot,\p\Omega) \quad\text{in }\Omega.\]
Using \eqref{IIS} and Theorem \ref{thm1}(ii) repeatedly, we obtain $u_{k}\in C^{0, 1}(\overline{\Omega})$ for all $k\geq m+ 2 = \lfloor (n+ 4)/2 \rfloor$. The corollary is proved.
  \end{proof}
  We indicate an application of the Lipschitz subsolution in Lemma \ref{Lipsub} in obtaining optimal boundary estimates for 
the Abreu's equation \cite{Ab} with degenerate boundary data which arises in the study of the existence of constant scalar curvature K\"ahler metrics for toric varieties \cite{D1, D2}; see also  \cite{CLS} for related 
Abreu's equation with degenerate boundary data.

\begin{thm} [Boundary Lipschitz estimates for the inverse of the Hessian determinant of Abreu's equation with degenerate boundary data]
\label{wphi}
Let $\Omega\subset\R^n$ be a bounded convex domain. Let $u\in C^{4}(\Omega)$ be a locally uniformly convex solution of 
\begin{equation}
\label{AMCE}
\left \{
\begin{alignedat}{2}
U^{ij}D_{ij}w&=-f&&\quad \text{in}~\Omega,\\
w &= [\det D^{2} u]^{-1}&&\quad \text{in}~\Omega,\\
w &= 0 &&\quad \text{on}~\p\Omega,
\end{alignedat}
\right.
\end{equation}
where $U = (U^{ij})=(\det D^{2} u) (D^{2} u)^{-1}$ is the cofactor matrix of the 
Hessian matrix $D^{2}u$, $f\in L^{\infty}(\Omega)$ with $f^{+}>0$, and $w\in C(\overline{\Omega})$.
 Then, we have the estimate
\begin{equation}
\label{Awlip}
\det D^2 u\geq c(n, \diam (\Omega)) \|f^{+}\|_{L^{\infty}(\Omega)}^{-n} [\dist (\cdot,\p\Omega)]^{-1}\quad\text{in }\Omega.
\end{equation}
\end{thm}

We refer the readers to \cite[Section 2]{D2} for a discussion of 
the optimality of \eqref{Awlip} when $\Omega$ is a polytope.

\begin{proof}
Note that for $\gamma>0$, $\tilde u = \gamma u$ satisfies
\begin{equation}
\label{u_res}
\tilde U^{ij} D_{ij} \tilde w = -\gamma^{-1} f,\end{equation}
where $\tilde U=(\tilde U^{ij})$ is the cofactor matrix of $D^2 \tilde u$, and $\tilde w = (\det D^2 \tilde u)^{-1}$. Thus, 
by considering $\|f^{+}\|_{L^{\infty}(\Omega)} u$ instead of $u$, it suffices to prove the theorem under the assumption that \[\|f^{+}\|_{L^{\infty}(\Omega)}=1.\]

Let $D:=\diam(\Omega)$. 
By
\cite[Theorem 1.6]{LCPAA}, we already have
\begin{equation}\det D^2 u(x)\geq c_1(n,D)  [\dist(x,\p\Omega)]^{- \frac{n-1}{n}}  \quad\text{for all } x\in\Omega.
\label{wstep2}
\end{equation}
Let $z=(z', z_n)$ be an arbitrary point in $\Omega$. We prove
\begin{equation}\det D^2 u(z)\geq 
c(n,D) [\dist(z,\p\Omega)]^{-1}.
\label{wstep2b}
\end{equation}
By translation and rotation of coordinates, we can assume that:  $0\in \p\Omega$,  $\Omega\subset \R^n_+$, 
 the $x_n$-axis points inward $\Omega$, $z$ lies on the $x_n$-axis so $z=(0, z_n)$, and $z_n=\dist (z,\p\Omega)$.

 To prove (\ref{wstep2b}), we will compare $w$ with $-C(a, D) v_{a, D}$ where \[a= \frac{2}{n} +  \frac{(n-1)^2}{n^2}>1,\] $v_{a, D}$ is the convex function in Lemma \ref{Lipsub}, and $C(a, D)>0$ is to be chosen.
 For this,
  we use
the fact that for $n\times n$ positive definite matrices $A$ and $B$, \[\text{trace}(AB)\geq n (\det A)^{1/n} (\det B)^{1/n}.\] 
Then, using (\ref{wstep2}) together with \[\det (U^{ij})= (\det D^2 u)^{n-1} \quad\text{and}\quad \det D^2 v_{a, D}\geq 2^{n-1}a  [\dist(x,\p\Omega)]^{na-2},\] we find that
\begin{eqnarray*}U^{ij} D_{ij}(-Cv_{a, D})& \leq& - Cn  (\det D^2 u)^{\frac{n-1}{n}} (\det D^2 v_{a, D})^{\frac{1}{n}}\\ &\leq& -C c_2(n, D)  [\dist(x,\p\Omega)]^{-\frac{(n-1)^2}{n^2}} [\dist(x,\p\Omega)]^{\frac{na-2}{n}}\\
&=&-  C c_2 (n, D)   <-1\leq  -f = U^{ij} D_{ij} w,
\end{eqnarray*}
if $C= C (n,D) $ is large. 

Since \[-U^{ij} D_{ij}(w+  C v_{a, D}) <0\quad \text{in } \Omega,\quad \text{and } w+  C v_{a, D} \leq 0 \quad \text{on } \p\Omega,\]
the classical comparison principle implies
\[w + C(a, D) v_{a, D}\leq 0 \quad\text{in }\Omega.\]
Therefore, at $z=(0, z_n)$, we have
 \[w(z) \leq -C(n,D) v_{a, D}(z)\leq  C(n,D) z_n =C(n,D) \dist(z,\p\Omega).\]
 This gives (\ref{wstep2b})  because $\det D^2 u(z)=[w(z)]^{-1}$.  The theorem is proved.
\end{proof}

  \section{Infinite boundary gradient}
  \label{pn2_sec}
  In this section, we prove Theorem \ref{pn2} using suitable subsolutions and supersolutions of the Monge--Amp\`ere equation  $\det D^2 u =|u|^{n-2}$.

  First, we construct subsolutions with $\dist(\cdot,\p\Omega)(1+ |\log \dist(\cdot,\p\Omega)|^{n/2})$ growth.
  \begin{lem}[Subsolutions]
  \label{logn2lem}
Assume that $0\in \p\Omega$ and $\Omega\subset \R^n_+$ ($n\geq 2$).
Let $D:=\diam(\Omega)$ and
\[ w(x):= \frac{x_n}{e^{2n} D} \bigg |\log \frac{x_n}{e^{2n} D}\bigg |^{\frac{n-2}{2}}\bigg(|x'|^2-D^2\bigg) -(1+ 4D^2)  \frac{x_n}{e^{2n} D} \bigg |\log \frac{x_n}{e^{2n} D}\bigg |^{\frac{n}{2}},\quad x\in\overline{\Omega}.\]
Then $w\in C^{\infty}(\Omega)$, $w$ is convex in $\Omega$ with $w\leq 0$ on $\p\Omega$, and there exists a constant $c=c(n,\diam(\Omega))>0$ such that 
\[\det D^2 w\geq c |w|^{n-2}\quad \text{in }\Omega.\]
\end{lem}
\begin{proof}
For $\alpha\in [0, n)$ and $t\in (0, e^{-2n})$, let
\[f_\alpha(t) := t (-\log t)^{\alpha}.\]
Then 
\[f_\alpha'(t) = (-\log t)^\alpha -\alpha (-\log t)^{\alpha-1},\quad f_\alpha''(t) = -\frac{\alpha}{t} (-\log t)^{\alpha-1}  -\frac{\alpha (1-\alpha)}{t} (-\log t)^{\alpha-2}. \]
Observe that
\[0<f_\alpha(t), \quad 0\leq f_\alpha'(t) \leq  |\log t|^\alpha, \quad  \frac{\alpha |\log t|^{\alpha-1}}{ 2t }<-f_\alpha''(t) \quad \text{for } t\in (0, e^{-2n}),\alpha\in [0, n).\]
Let
\[E =1+ 4D^2,\quad s= e^{2n} D,\quad y'= x'/s,\quad y_n=x_n/s\quad\text{for } x=(x', x_n)\in\Omega.\]
Consider
\[ u_{\alpha}(x):= f_\alpha(x_n/s) (|x'|^2-D^2) -E f_{\alpha+1} (x_n/s)\equiv f_\alpha(y_n) (|x'|^2-D^2) -E f_{\alpha+1} (y_n).\]
Then, $y_n\in (0, e^{-2n})$ and
\begin{equation}
\label{ualE}
|u_{\alpha}(x)| \leq  2E y_n |\log y_n|^{\alpha+1}\quad\text{in }\Omega.\end{equation}
We have
\begin{equation*}
 D^2 u_{\alpha}(x) =
  \begin{pmatrix}
  2f_\alpha(y_n) & 0 & \cdots & 0 & 2y_1 f_\alpha'(y_n) \\
  0 & 2f_\alpha(y_n) & \cdots & 0 & 2y_2  f_\alpha'(y_n) \\
  \vdots  & \vdots  & \ddots & \vdots & \vdots  \\
  0 & 0 & \cdots & 2 f_\alpha(y_n) &  2y_{n-1}  f_\alpha'(y_n)\\
  2y_1  f_\alpha'(y_n) & 2y_2  f_\alpha'(y_n) & \cdots & 2y_{n-1}  f_\alpha'(y_n) &  \frac{f_\alpha''(y_n) (|x'|^2-D^2) - Ef''_{\alpha+1} (y_n)}{s^2} 
 \end{pmatrix}.
\end{equation*}
Note that the first $(n-1)$-leading principal minors of $D^2 u_{(\alpha)}$ are all positive in $\Omega$.
We can compute by induction that 
\begin{equation*}
\begin{split}\det D^2 u_{\alpha}(x)&= 2^{n-1}s^{-2} [f_\alpha(y_n)]^{n-1} \big[ f_\alpha''(y_n) (|x'|^2-D^2)-E f''_{\alpha+1} (y_n)\big]\\ &\quad- 2^n s^{-2} |x'|^2 [ f_\alpha'(y_n)]^2 [ f_\alpha(y_n)]^{n-2}\\ &\geq  -2^{n-1} s^{-2} E   [f_\alpha(y_n)]^{n-1}  f_{\alpha+ 1}''(y_n)- 2^n s^{-2} D^2 |\log y_n|^{2\alpha} [ f_\alpha(y_n)]^{n-2}\\&\geq 2^{n-1} s^{-2} E   [y_n |\log y_n|^\alpha]^{n-1}  \frac{\alpha+1}{ 2y_n}|\log y_n|^\alpha\\&\quad- 2^n s^{-2} D^2 |\log y_n|^{2\alpha} [y_n |\log y_n|^\alpha]^{n-2}\\&= 2^{n-2} s^{-2} y_n^{n-2} |\log y_n|^{n\alpha} [(1+\alpha) E- 4D^2]\\&\geq 2^{n-2} s^{-2} y_n^{n-2} |\log y_n|^{n\alpha}.
\end{split}
\end{equation*}
Clearly, $u_{\alpha}\in C^{\infty}(\Omega)$, $u_{\alpha}$ is convex in $\Omega$ with $u_{\alpha}\leq 0$ on $\p\Omega$. Moreover,
when \[(\alpha+ 1)(n-2)= n\alpha,\quad\text{or } \alpha= (n-2)/2,\]
we obtain from the preceding inequalities and \eqref{ualE} that 
\[\det D^2 u_{(n-2)/2}\geq 2^{n-2} s^{-2} (y_n|\log y_n|^{n/2})^{n-2}\geq c(n,\diam(\Omega)) |u_{(n-2)/2}|^{n-2}\quad \text{in }\Omega,\]
for some constant $c(n,\diam(\Omega)) >0$. 

Since $w= u_{(n-2)/2}$, the lemma is proved.
\end{proof}

Next, we construct supersolutions with infinite boundary gradient. Note that these supersolutions are not necessarily convex. 
\begin{lem} [Supersolutions with infinite boundary gradient.]
\label{inflem}
Assume $0<x_n<e^{-1}$ for all $x\in\Omega\subset\R^n$ ($n\geq 2$). Let 
\[v(x)= x_n [(-\log x_n)^{1/n}-1] \big(|x'|^2-1\big),\quad x=(x', x_n)\in\Omega.\]
Then $v\in C^{\infty}(\Omega)$ and
\[\det D^2 v(x) \leq 2^n x_n^{n-2}\quad \text{in }\Omega.\]
\end{lem}
\begin{proof}
For $\alpha\in (0, 1)$ and $t\in (0, e^{-1})$, let 
\[g_\alpha(t)= t (-\log t)^{\alpha}-t.\]
Note that \[0<g_\alpha(t)< t|\log t|^\alpha, \quad 0<-g_\alpha''(t) \leq \frac{2\alpha |\log t|^{\alpha-1}}{ t }\quad \text{for all } t\in (0, e^{-1}),\,\alpha\in (0, 1).\]
Let
\[ v_{\alpha}(x):= g_\alpha(x_n) (|x'|^2-1)\quad \text{for } x\in\Omega.\]
Then, we can compute as in Lemma \ref{logn2lem} that
\begin{equation}
\label{log1n}
\begin{split}\det D^2 v_{\alpha}(x)&= 2^{n-1} [g_\alpha(x_n)]^{n-1}  g_\alpha''(x_n) (|x'|^2-1)- 2^n |x'|^2 [ g_\alpha'(x_n)]^2 [ g_\alpha(x_n)]^{n-2}\\ &\leq  -2^{n-1}  [g_\alpha(x_n)]^{n-1}  g_\alpha''(x_n)\\&\leq 2^n \alpha x_n^{n-2} |\log x_n|^{n\alpha-1}.\end{split}
\end{equation}
Thus, choosing $\alpha =1/n$, we have $v= v_{1/n}\in C^{\infty}(\Omega)$, and 
$\det D^2 v (x) \leq 2^n x_n^{n-2}$ in $\Omega$. The lemma is proved.
\end{proof}
\begin{proof}[Proof of Theorem \ref{pn2}]
We first prove part (i). Let $z=(z', z_n)$ be an arbitrary point in $\Omega$. By translation and rotation of coordinates, we can assume that:  $0\in \p\Omega$,  $\Omega\subset \R^n_+$, 
 the $x_n$-axis points inward $\Omega$, $z$ lies on the $x_n$-axis so $z=(0, z_n)$, and $z_n=\dist (z,\p\Omega)$.  Let $D=\diam (\Omega)$. 
 We  need to show that \begin{equation}
\label{ulog}
|u(z)| \leq C(n, D) z_n \bigg(1+ |\log z_n|^{\frac{n}{2}}\bigg). 
\end{equation}
Let $w$ be the convex function in Lemma \ref{logn2lem}. Recall that \[\det D^2 w\geq c(n, D)|w|^{n-2},\quad  \det D^2 u=|u|^{n-2}\quad \text{in }\Omega.\]
 From the consequence of the comparison principle in Lemma \ref{comp_lem}, we have 
\[|u|\leq C(n, D)|w| \quad\text{in }\Omega.\] 
 Since \[|w(z)|\leq C(n, D)z_n\bigg|\log \frac{z_n}{e^{2n}D}\bigg|^{n/2},\]
 we obtain \eqref{ulog}.

\medskip
Now, we prove part (ii). Fix $z\in \Omega$ being close to $\Gamma$. By translating and rotating coordinates, we can assume that for some $s=s(n, \Omega,\Gamma)\in (0, e^{-1})$,
\[z=(0, z_n)\in \Omega_{s}:=\{(x', x_n): |x'|<s, 0<x_n< s\}\subset\Omega,\] \[ \{(x', 0):|x'|\leq s\}\subset\Gamma\subset\p\Omega,
\quad\text{and }\dist(x,\p\Omega) = x_n\quad\text{for } x\in\Omega_s.\]
Let $g(t) = t (-\log t)^{1/n}-t$ and 
\[\bar w(x) = g(x_n/(es)) (|x'|^2 -s^2) \quad\text{in }\Omega_s.\]
Then, $\bar w\in C^{\infty}(\Omega_s)$. As in \eqref{log1n}, we have 
\begin{equation*}\begin{split}\det D^2 \bar w &\leq -2^{n-1} e^{-2}[g(x_n/(es))]^{n-1}  g''(x_n/(es)) \\ &\leq 2^n (es)^{2-n} x_n^{n-2}\\&= 2^n (es)^{2-n} [\dist(x,\p\Omega)]^{n-2} \quad\text{in }\Omega_s.\end{split}\end{equation*}
On the other hand, the convexity of $u$ and \eqref{upinfty} give
\begin{equation*}|u(x)|\geq \frac{\dist(x,\p\Omega)}{\diam (\Omega)}\|u\|_{L^{\infty}(\Omega)} \geq c(n,\Omega) \dist(x,\p\Omega).\end{equation*}
Thus, for some $R= R(n,s,\Omega)>0$, we have
\begin{equation}
\label{Ruw}
\det D^2 (Ru) =R^{n} |u|^{n-2}> \det D^2 \bar w\quad\text{in }\Omega_s. \end{equation}
Observe that 
\begin{equation}
\label{Ruw2}
Ru\leq \bar w\quad \text{in }\overline{\Omega_s}.
\end{equation}

 Indeed, since $\bar w=0$ on $\p\Omega_s$, this holds on $\p\Omega_s$. Suppose otherwise that $Ru-\bar w$ attains a positive maximum at $\bar x\in\Omega_s$. Then, $D^2 \bar w (\bar x)\geq RD^2 u(\bar x)$, as symmetric matrices. Now, the convexity of $u$ gives $RD^2 u(\bar x)\geq 0$ and $\det D^2 \bar w(\bar x) \geq \det D^2 (Ru) (\bar x)$, which contradicts \eqref{Ruw}. 
 
 \medskip
 \noindent
 Using \eqref{Ruw2} at $z=(0, z_n)\in\Omega_s$ where $0<z_n$ is small (due to $z$ being close to $\Gamma$), we have
\begin{equation*}
\begin{split}
|u(z)|\geq |\bar w(z)|/R&\geq c_1(n,\Gamma,\Omega) \dist (z,\p\Omega) \bigg (\bigg|\log \frac{\dist (z,\p\Omega)}{es}\bigg|^{\frac{1}{n}}-1\bigg)
\\& \geq c(n,\Gamma,\Omega) \dist (z,\p\Omega) |\log \dist (z,\p\Omega)|^{\frac{1}{n}}.
\end{split}
\end{equation*}
The theorem is proved.
\end{proof}
\begin{rem} The conclusion of Theorem \ref{pn2}(ii) still holds if $u\in C(\overline{\Omega})\cap C^2(\Omega)$ is a convex function satisfying
 \begin{equation*}
   \det D^{2} u~\geq  [\dist(\cdot,\p\Omega)]^{n-2} \quad \text{in} ~\Omega, \quad
u =0\quad \text{on}~\p \Omega.
\end{equation*}
\end{rem}
\medskip

{\bf Acknowledgements.} The author would like to thank the referee for carefully reading the paper and providing constructive comments.


\begin{thebibliography}{99999999}

\bibitem{AK} Abedin, F.; Kitagawa, J. Inverse iteration for the Monge-Amp\`ere eigenvalue problem. {\it Proc. Amer. Math. Soc.} {\bf 148} (2020), no. 11, 4875--4886.
\bibitem{Ab} Abreu, M. K\"ahler geometry of toric varieties and extremal metrics. {\it Inter. J. Math.} {\bf 9} (1998), 641--651.

\bibitem{C} Caffarelli, L. A. A localization property of viscosity solutions to the Monge--Amp\`ere equation and their strict convexity. {\it Ann. of Math.} (2) {\bf 131} (1990), no. 1, 129--134. 

\bibitem{CLS}Chen, B.; Li, A-M.; Sheng, L.  The Abreu equation with degenerated boundary conditions. {\it J. Differential Equations} {\bf 252} (2012), no. 10, 5235--5259.

\bibitem{DPF} De Philippis, G.; Figalli, A.
$W^{2,1}$ regularity for solutions of the Monge--Amp\`ere equation.
{\it Invent. Math.} {\bf 192}(2013), no.1, 55--69.
\bibitem{DPFS} De Philippis, G.; Figalli, A.; Savin, O.
A note on interior $ W^{2, 1+\varepsilon} $ estimates for the Monge--Amp\`ere equation.
{\it Math. Ann.} {\bf 357}(1) (2013), 11--22.

\bibitem{D1}Donaldson, S. K. Scalar curvature and stability of toric varieties.  {\it J. Differential Geom.}  {\bf 62}  (2002),  no. 2, 289--349.
\bibitem{D2} Donaldson, S. K. Interior estimates for solutions of Abreu's equation.  {\it Collect. Math.}  {\bf 56}  (2005),  no. 2, 103--142.


\bibitem{HHW} Hong, J.; Huang, G.; Wang, W. Existence of global smooth solutions to Dirichlet problem for degenerate elliptic Monge--Amp\`ere equations. 
{\it Comm. Partial Differential Equations.} {\bf 36} (2011), no. 4, 635--656.

\bibitem{HL} Huang, G.; L\"u, Y.
Analyticity of the solutions to degenerate Monge-Amp\`ere equations.
{\it J. Differential Equations.} {\bf 376}(2023), 633--654.
 \bibitem{LSNS} Le, N. Q. The eigenvalue problem for the Monge--Amp\`ere operator on general bounded convex domains.
 {\it Ann. Sc. Norm. Super. Pisa Cl. Sci.} (5) {\bf 18} (2018), no. 4, 1519--1559. 
\bibitem{LDCDS}Le, N. Q. Optimal boundary regularity for some singular Monge--Amp\`ere equations on bounded convex domains. {\it Discrete Contin. Dyn. Syst.}. {\bf 42} (2022), no. 5, 2199--2214.
\bibitem{LCPAA} Le, N. Q. Remarks on sharp boundary estimates for singular and degenerate Monge--Amp\`ere equations. {\it Commun. Pure Appl. Anal.} {\bf 22} (2023), no. 5, 1701--1720.
\bibitem{Lbook} Le, N. Q. {\em Analysis of Monge--Amp\`ere equations.}
Grad. Stud. Math., 240
American Mathematical Society, Providence, RI, [2024], ©2024. xx+576 pp.
\bibitem{LPDEA}Le, N. Q.  On global $W^{2,\delta}$ estimates for the Monge-Amp\`ere equation on general bounded convex domains,  {\it Partial Differ. Equ. Appl.}{\bf 5} (2024), no.4, Paper No.21, 10pp.
\bibitem{LArx}Le, N. Q. Convergence of an iterative scheme for the Monge--Amp\`ere eigenvalue problem with general initial data, arXiv:2006.06564v3. (Unpublished notes.)
\bibitem{LS} Le, N. Q.; Savin, O. Schauder estimates for degenerate Monge--Amp\`ere equations and smoothness of the eigenfunctions. {\it Invent. Math.} {\bf 207} (2017), no. 1,
389--423.
\bibitem{Ln} Lions, P. L. Two remarks on Monge--Amp\`ere equations. {\it Ann. Mat. Pura Appl.} {\bf 142} (1986) 263--275.
\bibitem{LLQ} Liu, H.; Leung, S.; Qian, J. An Efficient Operator-Splitting Method for the Eigenvalue Problem of the Monge--Amp\`ere Equation. 
{\it Commun. Optim. Theory.} {\bf 2022} (2022), pp. 1--22.
\bibitem{S} Savin, O. A localization theorem and boundary regularity for a class of degenerate Monge--Amp\`ere equations. {\it J. Differential Equations} {\bf 256} (2014), no. 2, 327--388. 
\bibitem{Tso}Tso, K. On a real Monge--Amp\`ere functional. {\it Invent. Math.} {\bf 101} (1990), no. 2, 425--448.



\end{thebibliography}
\end{document}